\PassOptionsToPackage{square,comma,numbers,sort&compress}{natbib}
\documentclass{article}

\usepackage[final]{neurips_2021}
\usepackage{bbding}
\usepackage[utf8]{inputenc} 
\usepackage[T1]{fontenc}    
\usepackage{url}            
\usepackage{booktabs}       
\usepackage{amsfonts, amsmath, amssymb, amsthm}       
\usepackage{nicefrac}       
\usepackage{microtype}      
\usepackage{wrapfig}
\usepackage{algorithm}
\usepackage{algorithmic}
\usepackage{graphicx}
\usepackage{color}
\usepackage{lineno}
\usepackage{multirow}
\usepackage{bm}
\usepackage{diagbox}
\usepackage{makecell}
\usepackage{subfigure}

\usepackage{comment}

\newtheorem{theorem}{Theorem}
\newtheorem{lemma}{Lemma}

\newtheorem{definition}{Definition}
\newtheorem{note}{Note}

\title{Feasible Pairings for Decentralized Integral Controllability of Non-Square Systems}

\author{Yuhao Tong$^1$, Steven W. Su $^{1,2}$ $^*$
	\AND
	\thanks{ The corresponding author.}
	\thanks{$^{1}$ College of Artificial Intelligence and Big Data for Medical Sciences, Shandong First Medical University (Shandong Academy of Medical Sciences), P. R. China.}
	\thanks{$^{2}$ University of Technology Sydney, Australia.}
}

\begin{document}
	
	\maketitle

\begin{abstract}
	This paper investigates the determination of feasible input-output pairings for the decentralized integral controllability of non-square systems. The relevance of this problem extends beyond traditional industrial processes into modern AI research, particularly Multi-Agent Reinforcement Learning (MARL), where environments frequently act as strongly non-square mappings that evaluate high-dimensional joint action spaces via comparatively low-dimensional global rewards. To address the stability of these complex distributed architectures, we extend the concept of D-stability to non-square matrices, providing a crucial mathematical foundation. We formally define D-stability for non-square matrices as a direct generalization of the square case. By introducing the concept of ``Squared Matrices'', which are derived from specific column selections of the non-square formulation and directly correspond to candidate control pairings, we establish a fundamental link between the stability of these square sub-components and the original non-square system. Ultimately, we propose sufficient conditions under which the individual Volterra-Lyapunov stability of these squared components guarantees the extended D-stability of the non-square matrix, thereby providing a rigorous method to identify feasible pairings that ensure robust decentralized control across both classical and data-driven applications.
\end{abstract}

\section{Introduction}

In various scientific and engineering applications, the stability of a system under structural variations and constraints is a primary concern. For instance, in the linearization of diffusion models for biological systems at constant equilibrium, the concept of strongly stable matrices has been proposed \cite{CROSS1978253, hadeler2006nonlinear}. Similarly, in systems control, decentralized stability has motivated extensive research into the properties of various matrix classes, including P-matrices, D-stable matrices, and Volterra-Lyapunov stable matrices \cite{sun2023gallery, Wang:-On, johnson1974sufficient, goh1977global, iggidr2023limits, barker1978positive}. These mathematical structures also find significant applications in economics.

In the process control industry, decentralized or distributed control architectures are often preferred over centralized schemes due to their simplicity, enhanced flexibility, and fault tolerance \cite{anderson1981algebraic, Anderson:-Alge, Sko:-ariab, grosdidier1986interaction, Su:-Analy,Su2015_DecentralizedIntegralControllability}. While the controllability of linear square systems under decentralized structures incorporating integral action is well-explored, challenges arise when dealing with non-square processes. Such non-square configurations---characterized by an unequal distribution of inputs and outputs---are prevalent in various control and optimization scenarios. For example, aerospace flight control systems often feature more actuators than primary outputs, robotic systems frequently possess redundant degrees of freedom, and certain chemical distillation columns rely on an unbalanced number of inputs and outputs to remain economically efficient.

Traditionally, controller design for these non-square processes involves transforming the system into a square configuration by either removing surplus inputs or adding new outputs. However, this forced squaring method can severely jeopardize the feedback system's flexibility and reliability. Specifically, the removal of surplus inputs diminishes the adaptability of the controller, while the addition of new outputs necessitates supplementary control efforts. 

To address these limitations, there is strong motivation to apply decentralized integral controllability directly to non-square processes (DIC-NSQ) without requiring structural transformations. Extending D-stability to non-square matrices provides the essential mathematical foundation to guarantee the stability of these distributed systems. Most importantly, this extended framework allows us to systematically evaluate and select {feasible input-output pairings} without discarding redundant actuators, ensuring independent tuning and robust performance even in the face of partial or complete actuator failures.

Recently, Multi-Agent Reinforcement Learning (MARL) has also emerged as a data-driven framework for distributed decision-making and cooperative control \cite{Busoniu2008_ComprehensiveSurveyMARL, Gronauer2022_DecentralizedMARLSurvey, Canese2021_MARLreview}. In many practical MARL formulations, the environment acts as a strongly non-square mapping: a high-dimensional joint action space (one control input per agent, often combinatorially large) is evaluated via a comparatively low-dimensional global state or team reward. From a control perspective, this parallels distributed control design for non-square systems where control channels are more than that of measured outputs. Recent works have explicitly utilized MARL to construct scalable distributed controllers for large-scale systems, such as networks and UAV swarms \cite{Zhang2020_LyapunovMARL, Gandhi2023_MARL_PowerSystem, Yang2023_GraphMeanFieldMARL}.

Given the prevalence of non-square systems in these diverse domains, addressing the stability of their associated real matrices is essential \cite{zhang2017multiloop, WHEATON20171, zhiteckii2022robust, sujatha2022control, steentjes2022data}. This necessitates the extension of stability concepts from square to non-square matrices to formalize how redundant inputs can be safely paired with primary outputs. While this paper focuses on real matrices, extensions to complex matrices are achievable via frequency-domain techniques. 

The primary contribution of this study is the identification of conditions that guarantee extended D-stability for non-square matrices, which in turn provides a rigorous theoretical basis for establishing $feasible$ $decentralized$ $pairings$. We derive a mild sufficient condition based on the Volterra-Lyapunov stability of the system's square sub-components. A natural open problem arising from this work is to characterise the necessary and sufficient conditions for D-stability of non-square matrices.

\section{Decentralized Stabilization for Non-Square Systems}


The motivation for extending D-stability to non-square matrices arises partly from Decentralized Unconditional Stability (DUS) analysis \cite{Sko:-ariab, grosdidier1986interaction, Su:-Analy, Su2025_DistributedIntegralControllability}. A key DUS condition for square systems mirrors the D-stability requirement of the steady-state gain matrix. Here, we generalize this to non-square systems.

\begin{figure}[htpb]
	\centering
	\includegraphics[width=0.9\linewidth]{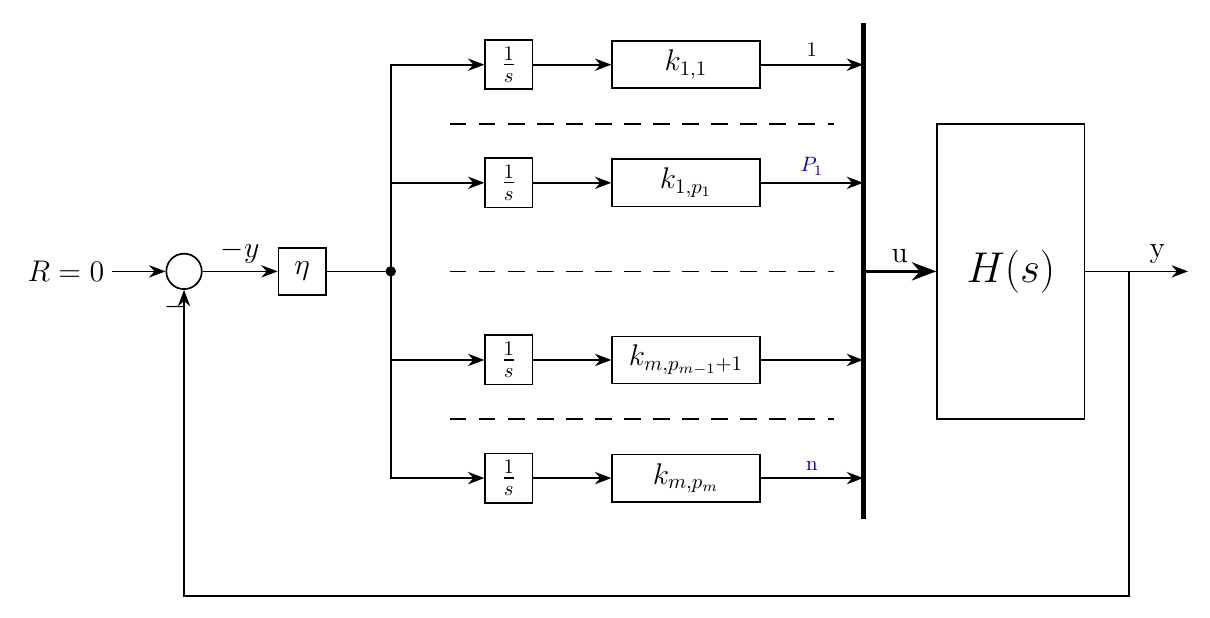}
	\caption{Distributed Integral Controllability for a non-square system. The process maps a high-dimensional input space to a lower-dimensional output space, accommodating redundant control channels.}
	\label{fig:dic_nsq}
\end{figure}

To facilitate the analysis of Decentralized Integral Controllability for Non-Square (DIC-NSQ) systems (see Figure \ref{fig:dic_nsq}), we utilize a state-space realization of the combined system and the decentralized integral controller. Assume $(A, B, C, D)$ is a minimal realization of the process $H(s)$. That is:
\begin{equation} \label{eq:state_space_H}
	H(s) : \begin{cases}
		\dot{z} = A z + B u \\
		y = C z + D u
	\end{cases}
\end{equation}
where $z$ represents the state vector of the plant.

Furthermore, we define a state-space realization for the integral controller as follows:
\begin{equation} \label{eq:state_space_K}
	\begin{cases}
		\dot{x} = -\eta y \\
		u = \bar{K} x
	\end{cases}
\end{equation}
where $x$ represents the integral states, $\eta$ is a positive scalar tuning parameter governing the integration rate, and $\bar{K} \in \mathbb{R}^{n \times m}$ is the integral coefficient matrix for the non-square system. 

The structure of $\bar{K}$ is defined to map the $m$ outputs to the $n$ inputs via a block-structured coefficient matrix:
\begin{equation} \label{eq:K_matrix_expanded}
	\bar{K} =
	\begin{bmatrix}
		k_{1,1} & \cdots & k_{1,p_1} & 0 & \cdots & 0 & \cdots & 0 & \cdots & 0 \\
		0 & \cdots & 0 & k_{2,1} & \cdots & k_{2,p_2} & \cdots & 0 & \cdots & 0 \\
		\vdots & \ddots & \vdots & \vdots & \ddots & \vdots & \ddots & \vdots & \ddots & \vdots \\
		0 & \cdots & 0 & 0 & \cdots & 0 & \cdots & k_{m,1} & \cdots & k_{m,p_m}
	\end{bmatrix}^T
\end{equation}
where the total number of inputs is given by $\sum_{i=1}^{m} p_i = n$. 

By treating the integral states $x_i$ as the primary system states, we can form an unforced closed-loop state equation. This specific formulation allows the closed-loop system dynamics to be transformed into a standard singular perturbation form, which isolates the slow controller dynamics from the fast plant dynamics.

Building upon the state-space realization of the combined system and controller, the unforced closed-loop state equations can be depicted as follows:
\begin{equation} \label{eq:unforced_closed_loop}
	\begin{cases}
		\frac{1}{\eta}\dot{x} = -D\bar{K}x - Cz \\
		\dot{z} = B\bar{K}x + Az
	\end{cases}
\end{equation}

To analyse the stability of this system, we transform Equation \eqref{eq:unforced_closed_loop} into a standard singular perturbation form. By introducing a slow time scale $\tau = \eta(t - t_0)$ such that $\tau = 0$ at $t = t_0$, the system dynamics become:
\begin{equation} \label{eq:singular_perturbation}
	\begin{cases}
		\frac{dx}{d\tau} = -D\bar{K}x - Cz \\
		\eta\frac{dz}{d\tau} = B\bar{K}x + Az
	\end{cases}
\end{equation}

Following standard singular perturbation analysis, the fast system dynamics (represented by the state $z$) will converge rapidly. Specifically, as the integration rate $\eta$ becomes sufficiently small ($\eta \to 0$), the fast dynamics reach a quasi-steady state where $B\bar{K}x + Az \approx 0$. Since the original plant matrix $A$ is asymptotically stable, the inverse $A^{-1}$ exists, allowing us to express the fast state as $z \approx -A^{-1}B\bar{K}x$.

Substituting this quasi-steady state back into the slow dynamics yields the reduced-order model:
\begin{equation} \label{eq:reduced_order}
	\begin{cases}
		\dot{x} \approx -(D - CA^{-1}B)\bar{K}x = -H(0)\bar{K}x \\
		z \approx -A^{-1}B\bar{K}x
	\end{cases}
\end{equation}

Consequently, according to singular perturbation theory, there exists a small positive scale $\epsilon$ such that whenever $0 < \eta \le \epsilon$, the overall closed-loop system is asymptotically stable if the eigenvalues of $-H(0)\bar{K}$ lie in the open left-half plane.

Then, according to the singular perturbation theory, there exist a small positive scale $\epsilon$, such that whenever $0 < \eta \le \epsilon$, the overall closed loop system is asymptotically stable if %
\begin{equation}
	\operatorname{Re}\{\sigma_i(-H(0)\bar{K})\} < 0 \quad \text{for } i \in \{1, 2, \dots, m\}
\end{equation} %

The preceding singular perturbation analysis establishes a fundamental link between the dynamic closed-loop stability of the system $H(s)$ and the algebraic properties of its steady-state gain matrix, $H(0)$. %
Specifically, the condition for the system to achieve Decentralized Unconditional Stability (DUS) is governed entirely by the eigenvalues of the matrix product $-H(0)\bar{K}$. %
Consequently, we can abstract the system-level Decentralized Integral Controllability (DIC) into a purely matrix-theoretic framework. %
By assuming a generic non-square real matrix $A$ is equivalent to the steady-state gain matrix $H(0)$, the problem of verifying DIC for the dynamical system simplifies directly to evaluating the extended D-stability of the constant matrix $A$. %
To rigorously analyse this connection and establish the necessary mathematical bounds, we now formalize the definition of D-stability for a non-square matrix. %

\begin{definition} \label{Dfn1}
Let $A \in \mathbb{R}^{m \times n}$ be a non-square matrix with $n > m$. We define $A$ as \textbf{D-stable} if there exists a block diagonal matrix $K \in \mathbb{R}^{n \times m}$ such that for all non-negative diagonal matrices $E \in \mathbb{R}^{n \times n}$ and all $j \in \mathcal{M}$:
\begin{equation}\label{suf_con}
	\operatorname{Re}\{\sigma_i ( [A E K]_j ) \} > 0,
\end{equation}
where $\mathcal{M}$ is the index set of all $k$-tuples of integers in the range $1,\dots,m$, $[M]_j$ denotes the $j$-th principal minor (or relevant sub-block structure defined by the context of decentralized control \cite{campo1994achievable}), and $\operatorname{Re}\{ \sigma_i (\cdot) \}$ denotes the real part of the eigenvalues.
\end{definition}

\begin{note} \label{note1}
Following \cite{campo1994achievable}, if diagonal entries of $E$ are zero, the corresponding columns of $A$ (and rows of $K$) are considered inactive and removed from the product $AEK$. If all inputs associated with a specific column block of $K$ are zero, that entire block is effectively deleted from the analysis.
\end{note}

Definition \ref{Dfn1} generalizes classical D-stability; in the square case ($n=m$), it reduces to the standard definition. To analyze this, we define the specific structure of the matrix $K$. Let $K$ be a block-diagonal mixing matrix defined as:
\begin{equation}\label{integral_matrix}
K =
\begin{bmatrix}
	\mathbf{k}_1 & \mathbf{0} & \cdots & \mathbf{0} \\
	\mathbf{0} & \mathbf{k}_2 & \cdots & \mathbf{0} \\
	\vdots & \vdots & \ddots & \vdots \\
	\mathbf{0} & \mathbf{0} & \cdots & \mathbf{k}_m
\end{bmatrix}^T \in \mathbb{R}^{n \times m},
\end{equation}
where each block $\mathbf{k}_i = [k_{i,1}, \dots, k_{i, p_i}]$ is a row vector of length $p_i$, and $\sum_{i=1}^{m} p_i = n$. For simplicity, we assume $k_{i,j} \ge 0$.

Let $E = \operatorname{diag}(\varepsilon_{1,1}, \dots, \varepsilon_{m,p_m})$ with $\varepsilon_{i,j} \ge 0$. The product $\bar{K} = EK$ preserves the block structure, scaling elements by $\varepsilon_{i,j}$. Partitioning $A$ compatible with $K$:
\[
A = [\boldsymbol{a}_{1,1}, \dots, \boldsymbol{a}_{1,p_1}, \mid \dots \mid, \boldsymbol{a}_{m,1}, \dots, \boldsymbol{a}_{m,p_m}],
\]
the product $AEK$ becomes a linear combination of columns within each block:
\begin{equation}\label{eq_li_comb}
AEK = \left[ \sum_{j=1}^{p_1} \varepsilon_{1,j} k_{1,j} \boldsymbol{a}_{1,j}, \;\; \dots, \;\; \sum_{j=1}^{p_m} \varepsilon_{m,j} k_{m,j} \boldsymbol{a}_{m,j} \right] \in \mathbb{R}^{m \times m}.
\end{equation}

\subsection{Squared Matrices and Volterra-Lyapunov Stability}

We now introduce the concept of ``Squared Matrices'' to bridge the gap between the non-square $A$ and square stability theories.

\begin{definition}[Squared Matrices] \label{squaredmatrix}
Consider $A \in \mathbb{R}^{m \times n}$ partitioned as above. We construct a set of $m \times m$ matrices, denoted as the \textbf{Squared Matrices} of $A$.
A squared matrix $[A]^m_{s_l}$ is formed by selecting exactly one column from each of the $m$ partitions of $A$.
Formally, let $\kappa = (\kappa_1, \dots, \kappa_m)$ be a tuple where $\kappa_i \in \{1, \dots, p_i\}$. The corresponding squared matrix is:
\[
[A]^m_{s_\kappa} = [\boldsymbol{a}_{1,\kappa_1}, \boldsymbol{a}_{2,\kappa_2}, \dots, \boldsymbol{a}_{m,\kappa_m}].
\]
The total number of such matrices is $N = \prod_{i=1}^{m} p_i$.
\end{definition}
Squared matrices of dimension $k < m$ can be defined similarly if specific partitions are deactivated (corresponding to zero entries in $E$, as per Note \ref{note1}).

Recall the standard stability definitions \cite{Cross:-Thre}:
\begin{definition}
A matrix $M \in \mathbb{R}^{n \times n}$ is:
\begin{enumerate}
	\item \textbf{D-stable} if $MD$ is stable (eigenvalues in open LHP) for all diagonal $D > 0$.
	\item \textbf{Volterra-Lyapunov stable} if there exists a diagonal $D > 0$ such that $MD + DM^T > 0$ (positive definite).
\end{enumerate}
\end{definition}

\begin{lemma}[\cite{Cross:-Thre}] \label{lemma2}
If $M$ is Volterra-Lyapunov stable, then $M$ is D-stable. Furthermore, any principal minor of a Volterra-Lyapunov stable matrix is also D-stable.
\end{lemma}

We extend this to the non-square case:

\begin{definition} \label{def_NSQ_lyp_ind}
A non-square matrix $A \in \mathbb{R}^{m \times n}$ is \textbf{Individually Volterra-Lyapunov Stable} if for every squared matrix $[A]^m_{s_l}$ (indexed by $l \in \{1, \dots, N\}$), there exists an individual positive diagonal matrix $D_l > 0$ such that:
\[
[A]^m_{s_l} D_l + D_l ([A]^m_{s_l})^T > 0.
\]
\end{definition}

\section{Main Results}

The core theoretical result of this paper establishes that individual Volterra-Lyapunov stability implies the extended D-stability for non-square matrices. This relies on a combinatorial result regarding the existence of weighting parameters.

\begin{lemma}[Combinatorial Existence Lemma] \label{lem:combinatorial}
Consider $m$ disjoint groups of indices, where group $\phi$ has size $p_\phi$. Let $\kappa = (\kappa_1, \dots, \kappa_m)$ represent a combination selecting one index from each group. Let there be a set of positive scalars $\lambda^\phi_\kappa > 0$ associated with each combination and each group.
Let $P_{\phi, j}$ be the total weighted payoff for the $j$-th element of group $\phi$, defined as:
\[
P_{\phi,j} := \sum_{\kappa:\,\kappa_\phi = j} \gamma_{\kappa}\,\lambda^\phi_{\kappa},
\]
where $\gamma_\kappa > 0$ is a global weight for combination $\kappa$.
Given any set of desired ratios $k_j^\phi > 0$ for $j \in \{1, \dots, p_\phi-1\}$ such that $P_{\phi, j+1} / P_{\phi, 1} = k_j^\phi$, there exists a set of positive global weights $\{\gamma_\kappa\}_{\kappa}$ such that all prescribed ratios are satisfied simultaneously.
\end{lemma}

\begin{proof}
The proof is constructive and provided in Appendix~A. It utilizes an inductive argument on the groups $\phi=1, \dots, m$ to solve for $\gamma_\kappa$.
\end{proof}

\subsection{Illustrative Example: $3 \times 2 \times 3$ Configuration} \label{sec:example}

To clarify the constructive procedure in the proof above, we present a detailed independent example. Consider $m=3$ groups with sizes $p_1=3$, $p_2=2$, and $p_3=3$. The total number of combinations is $N = 3 \times 2 \times 3 = 18$.

We denote the indices as $\kappa = (i, j, \ell)$ where $i \in \{1,2,3\}$, $j \in \{1,2\}$, and $\ell \in \{1,2,3\}$.
The objective is to find 18 strictly positive weights $\gamma_{(i,j,\ell)}$ to satisfy the following target ratios:
\begin{itemize}
\item \textit{Group 1 ($p_1=3$):} Ratios $k_1^1$ (between $i=2$ and $1$) and $k_2^1$ (between $i=3$ and $1$).
\item \textit{Group 2 ($p_2=2$):} Ratio $k_1^2$ (between $j=2$ and $1$).
\item \textit{Group 3 ($p_3=3$):} Ratios $k_1^3$ (between $\ell=2$ and $1$) and $k_2^3$ (between $\ell=3$ and $1$).
\end{itemize}

\noindent \textit{Step 1: Processing Group 1 (fixing tail $\alpha = (j, \ell)$)}

In the proof, the ``tail'' index is $\alpha$. Here, $\alpha = (j, \ell)$. We satisfy the ratios term-wise for each of the $2 \times 3 = 6$ tails.
For every pair $(j, \ell)$, we define $\gamma_{(2,j,\ell)}$ and $\gamma_{(3,j,\ell)}$ in terms of $\gamma_{(1,j,\ell)}$:
\begin{equation}
\gamma_{(2,j,\ell)} := k_1^1 \frac{\lambda^1_{(1,j,\ell)}}{\lambda^1_{(2,j,\ell)}} \gamma_{(1,j,\ell)}, \qquad
\gamma_{(3,j,\ell)} := k_2^1 \frac{\lambda^1_{(1,j,\ell)}}{\lambda^1_{(3,j,\ell)}} \gamma_{(1,j,\ell)}.
\end{equation}
This ensures that the payoff ratios for Group 1 are satisfied regardless of the value of $\gamma_{(1,j,\ell)}$.
Consistent with the inductive proof, the remaining free parameters are those where the first index is fixed to 1. We define these as the level-1 base parameters $\beta^{(1)}_{\alpha}$:
\begin{equation}
\beta^{(1)}_{(j,\ell)} := \gamma_{(1,j,\ell)}.
\end{equation}

\noindent \textit{Step 2: Processing Group 2 (fixing tail $\eta' = (\ell)$)}

We now satisfy the ratio $k_1^2$ for Group 2. The remaining tail index is $\eta' = (\ell)$.
The total payoff $P_{2,j}$ is a sum over indices $i$ and $\ell$. We group the terms by the tail $\ell$.
Using the substitutions from Step 1, the contribution of any specific tail $\ell$ to the payoff $P_{2,j}$ is a linear function of the base parameter $\beta^{(1)}_{(j,\ell)}$. Let $A_{2, j, \ell}$ be the accumulated coefficient for index $j$ and tail $\ell$ (corresponding to $A_{j, \eta'}$ in the proof):
\[
A_{2, j, \ell} := \lambda^2_{(1,j,\ell)} + k_1^1 \frac{\lambda^1_{(1,j,\ell)}}{\lambda^1_{(2,j,\ell)}} \lambda^2_{(2,j,\ell)} + k_2^1 \frac{\lambda^1_{(1,j,\ell)}}{\lambda^1_{(3,j,\ell)}} \lambda^2_{(3,j,\ell)}.
\]
Then the payoff is $P_{2,j} = \sum_{\ell} A_{2, j, \ell} \beta^{(1)}_{(j,\ell)}$. We define $\beta^{(1)}_{(2,\ell)}$ in terms of $\beta^{(1)}_{(1,\ell)}$ to satisfy the ratio $k_1^2$ term-wise:
\begin{equation}
\beta^{(1)}_{(2,\ell)} := k_1^2 \frac{A_{2, 1, \ell}}{A_{2, 2, \ell}} \beta^{(1)}_{(1,\ell)}.
\end{equation}
The remaining free parameters now correspond to fixing the first \emph{two} indices to 1 (i.e., $i=1, j=1$). We define the level-2 base parameters:
\begin{equation}
\beta^{(2)}_{(\ell)} := \beta^{(1)}_{(1,\ell)} = \gamma_{(1,1,\ell)}.
\end{equation}

\noindent \textit{Step 3: Processing Group 3 (no tail remaining)}

Finally, we satisfy the ratios $k_1^3$ and $k_2^3$ for Group 3.
The payoffs $P_{3,\ell}$ are sums over the previous indices $i, j$. Using the substitutions from Step 2, each payoff is a multiple of $\beta^{(2)}_{(\ell)}$.
Let $A_{3, \ell}$ be the total accumulated coefficient for index $\ell$ (computed similarly to Step 2 using the known $\lambda$'s, $k$'s, and $A$'s). We define:
\begin{equation}
\beta^{(2)}_{(2)} := k_1^3 \frac{A_{3, 1}}{A_{3, 2}} \beta^{(2)}_{(1)}, \qquad
\beta^{(2)}_{(3)} := k_2^3 \frac{A_{3, 1}}{A_{3, 3}} \beta^{(2)}_{(1)}.
\end{equation}
Noting that $\beta^{(2)}_{(\ell)} = \gamma_{(1,1,\ell)}$, this step expresses $\gamma_{(1,1,2)}$ and $\gamma_{(1,1,3)}$ in terms of $\gamma_{(1,1,1)}$.

\noindent \textit{Result}

By choosing any strictly positive value for $\gamma_{(1,1,1)}$ (e.g., $\gamma_{(1,1,1)}=1$), all other weights $\gamma_{(i,j,\ell)}$ are uniquely determined via backward substitution through Steps 3, 2, and 1. This construction guarantees that all 18 weights are positive and all target ratios are satisfied simultaneously.

\begin{theorem} \label{thm:main}
If a real non-square matrix $A \in \mathbb{R}^{m \times n}$ is Individually Volterra-Lyapunov stable (Definition \ref{def_NSQ_lyp_ind}), then $A$ is D-stable (Definition \ref{Dfn1}). That is, there exists a matrix $K$ such that for all non-negative diagonal $E$, $\operatorname{Re}\{ \sigma_i ( [A E K]_j ) \} > 0$.
\end{theorem}

\begin{proof}
Assume $E$ is strictly positive (cases with zero elements follow by dimensionality reduction as per Note \ref{note1}).
Since $A$ is individually Volterra-Lyapunov stable, for every squared matrix index $l \in \{1, \dots, N\}$, there exists $D_l > 0$ such that:
\begin{equation}\label{eq_mdl}
	[A]^m_{s_l} D_l + D_l ([A]^m_{s_l})^T > 0.
\end{equation}
Since the sum of positive definite matrices is positive definite, for any arbitrary coefficients $\gamma_l > 0$:
\begin{equation} \label{al_1}
	\sum_{l=1}^{N} \gamma_l \left( [A]^m_{s_l} D_l + D_l ([A]^m_{s_l})^T \right) > 0.
\end{equation}
Let $D_{sum} = \sum_{l=1}^{N} \gamma_l [A]^m_{s_l} D_l$. By standard properties of Lyapunov operators, the matrix $D_{sum}$ has eigenvalues with positive real parts.
We need to show that the matrix $AEK$ can be represented in the form of $D_{sum}$ (specifically, the column-weighted sum) by appropriate selection of $\gamma_l$.

From \eqref{eq_li_comb}, the $j$-th column of $AEK$ is $\sum_{r=1}^{p_j} \varepsilon_{j,r} k_{j,r} \boldsymbol{a}_{j,r}$.
The term $\sum_{l=1}^{N} \gamma_l [A]^m_{s_l} D_l$ can be expanded. Let $D_l = \operatorname{diag}(\lambda^1_l, \dots, \lambda^m_l)$. The $j$-th column of this sum is:
\[
\sum_{r=1}^{p_j} \left( \sum_{l: \kappa_j(l)=r} \gamma_l \lambda^j_l \right) \boldsymbol{a}_{j,r}.
\]
By setting the scalars in $AEK$ to match these sums, i.e., determining $\varepsilon_{j,r} k_{j,r}$ via the ratios of the inner sums, we invoke Lemma \ref{lem:combinatorial}. The lemma guarantees that for any desired ratios of column coefficients (determined by $E$ and $K$), there exist weights $\gamma_l$ such that the aggregate matrix matches $AEK$ up to a column scaling diagonal matrix (which does not affect sign stability). Thus, $AEK$ inherits the stability properties, proving the theorem.
\end{proof}

\section{Discussion and Open Problems}

The preceding analysis establishes a fundamental link between the dynamic closed-loop stability of the non-square system and the algebraic properties of its steady-state gain matrix, $H(0)$. Specifically, by strategically re-grouping and pairing the system's inputs and outputs, we can substantially change the effective structure of both $H(0)$ and the block-diagonal mixing matrix $K$. This configuration flexibility in the decentralized control architecture is highly advantageous, as it allows us to intentionally manipulate the system representation to satisfy our proposed sufficient conditions for DIC-NSQ, even if the default pairing does not.

To abstract this system-level Decentralized Integral Controllability (DIC) into a purely matrix-theoretic framework, we assume the generic non-square real matrix $A$ represents the negative steady-state gain matrix, such that $A = -H(0)$. By making this substitution, the problem of verifying DIC for the dynamical system simplifies directly to evaluating the extended D-stability of the constant matrix $A$.

In this study, we have provided mild sufficient conditions to guarantee the Decentralized Integral Controllability of non-square systems. We established that if the non-square matrix $A$ is an individually Volterra-Lyapunov stable matrix, or if all of its squared matrices are column strictly diagonally dominant, then the underlying non-square system achieves DIC-NSQ. 

However, establishing the exact boundaries of these stability requirements for both real matrices and dynamic systems remains an open mathematical problem. While our proposed sufficient conditions provide an efficient and practical method for controller design, the exact necessary and sufficient conditions for the extended D-stability of non-square matrices and systems are still lacking. Future research efforts should aim to fully characterise these necessary and sufficient conditions, which would further minimize conservatism in the design and analyse the absolute fundamental limits of non-square decentralized control systems.

\textbf{Open Problem: Necessary and Sufficient Conditions for Non-Square D-Stability}

Let $A \in \mathbb{R}^{m \times n}$ (with $n > m$) represent the negative steady-state gain matrix of a system, such that $A = -H(0)$. By strategically re-grouping and pairing the inputs and outputs, we establish the specific block-diagonal structure of the mixing matrix $K \in \mathbb{R}^{n \times m}$. 

The open problem is to determine the exact algebraic necessary and sufficient conditions on $A$ such that there exists at least one valid configuration of $K$ (with $k_{i,j} \ge 0$) ensuring the extended D-stability of $A$. Formally, this requires identifying the precise, non-conservative properties $A$ must possess so that for all non-negative diagonal matrices $E \in \mathbb{R}^{n \times n}$ and all index subsets $j \in \mathcal{M}$:
\begin{equation}
	\operatorname{Re}\{\sigma_i ( [A E K]_j ) \} > 0
\end{equation}

While sufficient conditions, such as the individual Volterra-Lyapunov stability of the squared matrices of $A$, have been established to guarantee Decentralized Integral Controllability (DIC-NSQ), the fundamental algebraic limits characterising the complete class of such stable non-square matrices remain unresolved.

%

\section{Conclusion}
Our primary contribution is the derivation of mild sufficient conditions that guarantee extended D-stability for non-square matrices, providing a rigorous mathematical foundation for identifying feasible input-output pairings. Specifically, we proved that if the set of all square matrices derived from the non-square system, which directly represent candidate control pairings, are individually Volterra-Lyapunov stable, or if they are column strictly diagonally dominant, the system is D-stable in the extended sense. This provides a robust and practical sufficient condition for the stability of non-square decentralized control systems, ensuring unconditional stability and fault tolerance without the need to artificially square the system by removing critical redundant actuators.

By directly addressing the structural challenges of non-square mappings, this framework extends beyond traditional industrial processes, offering valuable theoretical insights for modern data-driven architectures such as Multi-Agent Reinforcement Learning (MARL). Despite these advancements, establishing the exact boundaries for non-square D-stability remains an open mathematical problem. Future work will focus on identifying the exact necessary and sufficient conditions for the D-stability of non-square matrices. Additionally, further research will explore optimal strategies for structuring the block-diagonal mixing matrix $K$ through strategic input-output pairing, aiming to minimize conservatism and define the absolute fundamental limits of non-square decentralized control architectures across both classical and AI-driven environments.

%

\appendix

\section{Proof of Lemma \ref{lem:combinatorial}}

\begin{proof}
The proof is constructive. We demonstrate the existence of strictly positive weights $\gamma_\kappa$ by satisfying the ratio constraints group by group, sequentially from $\phi=1$ to $\phi=m$.

\subsubsection*{Setup and Definitions}
Fix $m \in \mathbb{N}$ and group sizes $p_\phi$. Let $\kappa = (\kappa_1, \dots, \kappa_m)$ be the multi-index for the combinations, where $\kappa_\phi \in \{1, \dots, p_\phi\}$.
We are given strictly positive proportionality factors $\lambda^\phi_\kappa > 0$ and target ratios $k_j^\phi > 0$.
Our goal is to find $\gamma_\kappa > 0$ such that for every group $\phi$ and every index $j \in \{1, \dots, p_\phi-1\}$:
\begin{equation} \label{eq:target_ratio}
	\frac{P_{\phi,j+1}}{P_{\phi,1}} = \frac{\sum_{\kappa:\,\kappa_\phi=j+1} \gamma_\kappa \lambda^\phi_\kappa}{\sum_{\kappa:\,\kappa_\phi=1} \gamma_\kappa \lambda^\phi_\kappa} = k_j^\phi.
\end{equation}

\subsubsection*{Processing Group 1}
Consider the first group ($\phi=1$). The total index can be split into the first component $\kappa_1$ and the ``tail'' $\alpha = (\kappa_2, \dots, \kappa_m)$.
The condition for Group 1 requires that the ratio of weighted sums equals $k_j^1$. We satisfy this condition in the strongest possible way: term-wise for every tail configuration.

Fix an arbitrary tail $\alpha$. We define the weight $\gamma_{(j+1, \alpha)}$ (where the first index is $j+1$) in terms of the weight $\gamma_{(1, \alpha)}$ (where the first index is $1$) as follows:
\begin{equation} \label{eq:base_definition}
	\gamma_{(j+1, \alpha)} := k_j^1 \frac{\lambda^1_{(1, \alpha)}}{\lambda^1_{(j+1, \alpha)}} \gamma_{(1, \alpha)}.
\end{equation}
Since $\gamma_{(1, \alpha)}$ is undetermined, we treat it as a free parameter. Let us verify the summation ratio:
\[
\frac{P_{1,j+1}}{P_{1,1}} = \frac{\sum_{\alpha} \gamma_{(j+1, \alpha)} \lambda^1_{(j+1, \alpha)}}{\sum_{\alpha} \gamma_{(1, \alpha)} \lambda^1_{(1, \alpha)}}.
\]
Substituting the definition from \eqref{eq:base_definition} into the numerator:
\begin{align*}
	\sum_{\alpha} \left( k_j^1 \frac{\lambda^1_{(1, \alpha)}}{\lambda^1_{(j+1, \alpha)}} \gamma_{(1, \alpha)} \right) \lambda^1_{(j+1, \alpha)}
	&= \sum_{\alpha} k_j^1 \left( \gamma_{(1, \alpha)} \lambda^1_{(1, \alpha)} \right) \\
	&= k_j^1 \sum_{\alpha} \gamma_{(1, \alpha)} \lambda^1_{(1, \alpha)}.
\end{align*}

The numerator is exactly $k_j^1$ times the denominator. Thus, the ratio holds.

At the end of this step, the independent degrees of freedom are the weights where the first index is fixed to 1. We define these as our first set of base parameters:
\[
\beta^{(1)}_{\alpha} := \gamma_{(1, \alpha)}.
\]
Note that $\alpha$ represents the indices $(\kappa_2, \dots, \kappa_m)$.

\subsubsection*{Inductive Step}
\emph{Hypothesis:} Assume we have processed groups $1, \dots, \delta$. Every weight $\gamma_\kappa$ is currently expressed as a specific positive coefficient $C^{(\delta)}_\kappa$ multiplied by a base parameter $\beta^{(\delta)}_{\eta}$, where $\eta = (\kappa_{\delta+1}, \dots, \kappa_m)$ is the remaining tail index.
Crucially, consistent with the base case, this base parameter corresponds to the weight where the first $\delta$ indices are fixed to 1:
\begin{equation}
	\beta^{(\delta)}_{\eta} = \gamma_{(1, \dots, 1, \eta)}.
\end{equation}

\emph{Processing Group $\delta+1$:}
We now satisfy the ratios for group $\phi = \delta+1$. Let the index for this group be $j = \kappa_{\delta+1}$. Let the remaining indices be the new tail $\eta' = (\kappa_{\delta+2}, \dots, \kappa_m)$.
We can write the base parameter $\beta^{(\delta)}_{\eta}$ explicitly as $\beta^{(\delta)}_{(j, \eta')}$.

The payoff sum $P_{\delta+1, j}$ is a sum over all possible combinations. We group this sum by the new tail $\eta'$:
\[
P_{\delta+1, j} = \sum_{\eta'} \left( \text{Terms depending on } \beta^{(\delta)}_{(j, \eta')} \right).
\]
Specifically, the contribution of a specific tail $\eta'$ to the payoff $P_{\delta+1, j}$ is a linear function of $\beta^{(\delta)}_{(j, \eta')}$. Let $A_{j, \eta'}$ be the accumulated positive coefficient for this term (derived from the $\lambda$'s and previous $k$'s).

To satisfy the target ratio $k_j^{\delta+1}$, we enforce the condition term-wise for each tail $\eta'$. We define the parameter $\beta^{(\delta)}_{(j+1, \eta')}$ (where index $\delta+1$ is $j+1$) in terms of $\beta^{(\delta)}_{(1, \eta')}$ (where index $\delta+1$ is $1$):
\begin{equation}
	\beta^{(\delta)}_{(j+1, \eta')} := k_j^{\delta+1} \frac{A_{1, \eta'}}{A_{j+1, \eta'}} \beta^{(\delta)}_{(1, \eta')}.
\end{equation}
This ensures that for every sub-group of terms defined by $\eta'$, the ratio is exactly $k_j^{\delta+1}$. Consequently, the ratio of the total sums is also $k_j^{\delta+1}$.

\subsubsection*{Summation}
By this construction, the independent parameters for the next step are those where the index for group $\delta+1$ is fixed to 1. We define the new base parameters:
\[
\beta^{(\delta+1)}_{\eta'} := \beta^{(\delta)}_{(1, \eta')} = \gamma_{(1, \dots, 1, 1, \eta')}.
\]
Iterating this process until $\phi=m$, the only remaining free parameter is the one where all indices are fixed to 1:
\[
\beta^{(m)} = \gamma_{(1, 1, \dots, 1)}.
\]
Since all proportionality constants generated in the steps are strictly positive (products/quotients of positive $\lambda$'s, $k$'s, and coefficients), choosing $\gamma_{(1, \dots, 1)} = 1$ (or any positive real number) uniquely determines all $\gamma_\kappa > 0$, satisfying all ratio constraints simultaneously.

The fully expanded forms of the aforementioned ratios, which include the original extended numerators and denominators, as well as the associated proof of this lemma, can be found in \bf{\cite{tong2024sufficient} for the sake of completeness}.
\end{proof}


\bibliography{VO2.bib}

\end{document}